\documentclass[12pt]{article}
\usepackage{hyperref}
\usepackage{cmap}                        % Ïîääåðæêà ïîèñêà ðóññêèõ ñëîâ â PDF (pdfla\TeX)
\usepackage[cp1251]{inputenc}            % Âûáîð ÿçûêà è êîäèðîâêè
\usepackage[english]{babel}
\usepackage[left=1in,right=1in,top=1in,bottom=1.5in]{geometry} % ïîëÿ ñòðàíèöû
\usepackage{amssymb,amsmath, amsthm, amscd,ifthen}
\usepackage{graphicx}

\newtheorem*{thm*}{Theorem}

\newcommand{\ff}{{\mathcal F}}

\newtheorem{thm}{Theorem}

\newtheorem{opr}{Definition}
\newtheorem{lem}[thm]{Lemma}
\newtheorem{cla}[thm]{Claim}
\newtheorem{thrm}{Theorem}
\newtheorem{cor}[thm]{Corollary}
\date{}

\newtheorem{prop}[thm]{Proposition}

\newcounter{tmp}

\title{Intersection theorems for $\{0,\pm1\}$-vectors and $s$-cross-intersecting families}
\author{Peter Frankl, Andrey Kupavskii\footnote{Moscow Institute of Physics and Technology, \'Ecole Polytechnique F\'ed\'erale de Lausanne; Email: {\tt kupavskii@yandex.ru} \ \ Research supported in part by the Swiss National Science Foundation Grants 200021-137574 and 200020-14453 and by the grant N 15-01-03530 of the Russian Foundation for Basic Research.}}

\date{}
\begin{document}
\maketitle
\begin{abstract} In this paper we pursue two possible directions for the extension of the classical Erd\H os-Ko-Rado theorem which states that any family of $k$-element, pairwise intersecting subsets of the set $[n] := \{1,\ldots,n\}$ has cardinality at most ${n-1\choose k-1}$.

In the first part of the paper we study families of $\{0,\pm 1\}$-vectors. Denote by $\mathcal L_k$ the family of all vectors $\mathbf v$ from $\{0,\pm 1\}^n$ such that $\langle\mathbf v,\mathbf v\rangle = k$. For any $k$, most $l$ and sufficiently large $n$, as well as for small $k$ and any $l,n$, we determine the maximal size of a family $\mathcal V\subset \mathcal L_k$ satisfying  $\langle \mathbf v,\mathbf w\rangle\ge l$ for any $\mathbf v,\mathbf w\in \mathcal V$.

In the second part of the paper we study cross-intersecting pairs of families. We say that two families $\mathcal A, \mathcal B$ are \textit{$s$-cross-intersecting}, if for any $A\in\mathcal A,B\in \mathcal B$ we have $|A\cap B|\ge s$. We also say that a family $\mathcal A$ of sets is \textit{$t$-intersecting}, if for any $A_1,A_2\in \mathcal A$ we have $|A_1\cap A_2|\ge t$. For a pair of nonempty $s$-cross-intersecting $t$-intersecting families $\mathcal A,\mathcal B$ of $k$-sets,  we determine the maximal value of $|\mathcal A|+|\mathcal B|$ for sufficiently large $n$ .
\end{abstract}

 \setcounter{tmp}{\value{thrm}}% store current value of theorem counter
 \setcounter{thrm}{0} %assign desired value to theorem counter
  \renewcommand\thethrm{\Alph{thrm}}
\section{Introduction}
Let $[n]= \{1,\ldots, n\}$ be an $n$-element set and for $0\le k\le n$ let ${[n]\choose k}$ denote the collection of all its $k$-element subsets. Further, let $2^{[n]}$ denote the power set of $[n]$. Any subset $\mathcal F\subset 2^{[n]}$ is called a \textit{set family} or \textit{family} for short. We say a ``$k$-set'' instead of a ``$k$-element set'' for shorthand. If $F\cap F'\ne \emptyset$ for all $F,F'\in\mathcal F$, then $\mathcal F$ is called \textit{intersecting}. The following classical result is one of the cornerstones of extremal set theory.

\begin{thm*}[Erd\H os-Ko-Rado theorem \cite{EKR}] If $\mathcal A\subset{[n]\choose k}$ is intersecting and $n\ge 2k$, then $|\mathcal A|\le {n-1\choose k-1}$ holds.
\end{thm*}
This is one of the first results in extremal set theory and probably the first result about intersecting families.

Numerous results extended the Erd\H os-Ko-Rado theorem in different ways. One of the directions is concerned with \textit{nontrivial} intersecting families. We say that an intersecting family is {\it nontrivial} if the intersection of {\it all} sets from the family is empty. Note that the size of a family of all $k$-sets containing a single element matches the bound from the Erd\H os-Ko-Rado theorem. Probably, the best known result in this direction is the Hilton-Milner theorem \cite{HM}, which determines the maximum size of a non-trivial intersecting family.

In the same paper Hilton and Milner dealt with pairs of \textit{cross-intersecting} families. We say that the families $\mathcal A,\mathcal B\subset 2^{[n]}$ are \textit{cross-intersecting} if for every $A\in\mathcal A$ and $B\in \mathcal B$ we have $A\cap B\ne \emptyset$. Hilton and Milner proved the following inequality:
\begin{thm*}[Hilton and Milner \cite{HM}] Let $\mathcal A,\mathcal B\subset{[n]\choose k}$ be non-empty cross-intersecting families with $n\ge 2k$. Then $|\mathcal A|+|\mathcal B|\le {n\choose k}-{n-k\choose k}+1.$
\end{thm*}
This inequality was generalized by Frankl and Tokushige \cite{FT} to the case $\mathcal A\subset{[n]\choose a}$, $\mathcal B\subset{[n]\choose b}$ with $a\ne b$, as well as to more general constraints on the sizes of $\mathcal A, \mathcal B$. A simple proof of the theorem above may be found in \cite{FK2}.

 Another natural  direction for the generalizations deals with more general constraints on the sizes of pairwise intersections. A family $\ff$ is called \textit{$t$-intersecting} if for any $F_1,F_2\in \ff$ we have $|F_1\cap F_2|\ge t$. The Complete Intersection Theorem due to Alswede and Khachatrian \cite{AK}, continuing previous work done by Frankl \cite{F1} and Frankl and F\" uredi \cite{FF}, gave an exhaustive answer to the question of how big a $t$-intersecting family of $k$-sets could be. We discuss it in more detail in Section \ref{sec3}.

One more example of a result of such kind is a  theorem due to Frankl and Wilson \cite{FW}. It states that, if $n$ is a prime power and $\mathcal F\subset {[4n-1]\choose 2n-1}$ satisfies $|F_1\cap F_2|\ne n-1$ for any $F_1,F_2\in \ff$, then the size of $\mathcal F$ is at most ${4n-1\choose n-1}$, which is exponentially smaller than ${4n-1\choose 2n-1}$. This theorem has several applications in geometry.

The third direction is to generalize the EKR theorem to more general classes of objects. Out of the numerous generalizations of such kind we are going to talk only about the ones concerned with $\{0,\pm 1\}$-vectors, that is, vectors that have coordinates from the set $\{-1,0,1\}$. In this case a natural replacement for the size of the intersection is  the scalar product. A generalization of \cite{FW} to the case of $\{0,\pm 1\}$-vectors due to Raigorodskii \cite{Rai1} led to improved bounds in several geometric questions, such as Borsuk's problem and the chromatic number of the space. However, it seems plausible to improve Raigorodskii's result, once one achieves a better understanding of ``intersecting'' families of $\{0,\pm 1\}$-vectors. Partly motivated by that, in our recent work \cite{FK,FK11} we managed to extend EKR theorem to the case of $\{0,\pm 1\}$-vectors with fixed number of $1$'s and $-1$'s. Let us formulate the results. We need the following notation.

Let $s\ge t\ge 1$ be integers and let $\mathcal V(n,s,t)$ denote the set of all $(0,\pm 1)$-vectors of length $n$ and having exactly $s$ coordinates equal to $+1$ and $t$ coordinates equal to $-1$. Put $$g(n,s,t):=\max\Big\{|\mathcal V|: \mathcal V\subset \mathcal V(n,s,t)\text{ and } \langle\mathbf v,\mathbf w\rangle> -2t\text{ for any }\mathbf v,\mathbf w\in \mathcal V\Big\}.$$
Note that, for $n\ge 2s$, $-2t$ is the smallest possible scalar product between two vectors in $\mathcal V(n,s,t)$: indeed, to achieve it, both vectors have to have their $-1$'s in front of the $1$'s of the other vector, and the remaining $s-t$ $1$'s  in front of $0$'s.
 \begin{thrm}[Frankl, Kupavskii \cite{FK}]\label{fkjcta} We have
\begin{align*}g(n,s,1)=&  s{n-1\choose s} \text{\ \ \qquad \qquad \quad \ \ \ \ \ for\ }2s\le n\le s^2,\\
g(n,s,1)=& s{s^2-1\choose s}+{s^2\choose s}+{s^2+1\choose s}+\ldots +{n-1\choose s}  \text{\ \ \ \ \ \ \ \ for\ } n> s^2.
\end{align*}
\end{thrm}

\begin{thrm}[Frankl, Kupavskii \cite{FK11}]\label{fkbar} We have
\begin{equation}\label{eq000}{n\choose s+t}{s+t-1\choose t-1}\le g(n,s,t)\le {n\choose s+t}{s+t-1\choose t-1}+{n\choose 2t}{2t\choose t}{n-2t-1\choose s-t-1}.\end{equation}
\end{thrm}

The contribution of this paper is two-fold. First, in several scenarios we determine  the maximal size of a family of $\{0,\pm 1\}$-vectors of fixed length and with restrictions on the scalar product. The main results are stated in Section \ref{sec22}. Second, we prove results analogous to the Hilton-Milner theorem stated above, but for pairs of families that are $s'$-cross-intersecting and $t'$-intersecting. For details and precise definitions see Section \ref{sec3}.

\section{Families of $\{0,\pm 1\}$-vectors}
Denote by $\mathcal L_k$ the family of all vectors $\mathbf v$ from $\{0,\pm 1\}^n$ such that $\langle\mathbf v,\mathbf v\rangle = k$. Note that $|\mathcal L_k| = 2^k{n\choose k}$.
This section is mostly devoted to the study of the quantity below.
\begin{equation}\label{eq2} F(n,k,l) := \max\{|\mathcal V|:\mathcal V\subset \mathcal L_k, \forall \, \mathbf v,\mathbf w\in\mathcal V\ \  \langle\mathbf v,\mathbf w\rangle\ge l\}.
\end{equation}

Recall the following theorem of Katona:

\begin{thm*}[Katona, \cite{K}]\label{thmkat} Let $n>s>0$ be fixed integers. If $\mathcal U\subset 2^{[n]}$ is a family of sets such that for any $U,V\in \mathcal U$ we have $|U \cup V|\le s$ then
\begin{equation}\label{eq1} |\mathcal U|\le f(n,s) := \left\{\begin{split}\sum_{i=0}^{s/2} {n\choose i}\ \ \ \ \ \ \ \ \ \ \ \  \text{if } s\text{ is even},\\ 2\sum_{i=0}^{(s-1)/2} {n-1\choose i}\ \ \ \ \text{if } s\text{ is odd.}\end{split}\right.
\end{equation}
Moreover, for $n\ge s+2$ the equality is attained only for the following families. If $s$ is even, than it is the family $\mathcal U^s$ of all sets of size at most $s/2$. If $s$ is odd, then it is one of the families $U^s_j$ of all sets that intersect $[n]-\{j\}$ in at most $(s-1)/2$ elements,  where $1\le j\le n$.
\end{thm*}

Given two sets $U,V$,  we denote the symmetric difference of these two sets by $U\triangle V$, that is, $U\triangle V:=U\backslash V\cup V\backslash U$. A theorem due to Kleitman states that the bound (\ref{eq1}) holds for a more general class of families.

\begin{thm*}[Kleitman, \cite{Kl}] \label{thmkl} If for any two sets $U,V$ from a family $\mathcal U\subset 2^{[n]}$ we have  $|U \triangle V|\le s,$ then the bound (\ref{eq1}) holds for $\mathcal U$.
\end{thm*}
Note that there is no uniqueness counterpart in Kleitman's theorem.

\subsection{Simple properties of $F(n,k,l)$}
First we state and prove some simple observations concerning $F(n,k,l)$.
\begin{prop}\label{prop1} Fix any $n\ge k\ge 1$.   Then $F(n,k,-k+1) = 2^{k-1}{n\choose k} = |\mathcal L_k|/2.$
\end{prop}

\begin{proof} We split vectors from $\mathcal L_k$ in pairs $\mathbf v, \mathbf w$ so that $\langle\mathbf v,\mathbf w\rangle = -k$. We can take exactly one vector out of each pair in the family.
\end{proof}
For any $\mathbf v = (v_1,\ldots, v_n)$ and any vector family $\mathcal V$ define $S(\mathbf v) = \{i:v_i\ne 0\}$ and $\mathcal V(S) = \{\mathbf v\in\mathcal V: S(\mathbf v) = S\}$. We also define $N(\mathbf v) = \{i:v_i = -1\}$.

\begin{prop}\label{prop2} We have $F(2t,2t,2s) = F(2t,2t,2s-1) = f(2t,t-s)$ and $F(2t+1,2t+1,2s) = F(2t+1,2t+1,2s+1) = f(2t+1,t-s).$
\end{prop}
\begin{proof} For any $\mathbf v,\mathbf w\in \mathcal L_{2t}$ we have $\langle\mathbf v,\mathbf w\rangle= 2t-2|N(\mathbf v)\triangle N(\mathbf w)|$. Therefore, the statement of the proposition follows from Kleitman's Theorem.
\end{proof}
We say that a family $\mathcal V\subset \mathcal L_k$ of vectors is \textit{homogeneous}, if for every $i\in[n]$ the $i$'th coordinates of vectors from $\mathcal V$ are \textit{all} non-negative or \textit{all} non-positive.
\begin{prop}\label{prop3} For $n>k$ we have $F(n,k,k-1) = \max\{k+1,n-k+1\}$.
\end{prop}
\begin{proof} First of all note that $u_iv_i = -1$ forces $\langle\mathbf u,\mathbf v\rangle\le k-2$. Thus $\mathcal V$ is homogeneous and the condition translates into $|S(\mathbf u)\cap S(\mathbf v)|\ge k-1$.

Since the only $(k-1)$-intersecting families of $k$-sets are  stars around a $(k-1)$-element set and subsets of ${[k+1]\choose k}$, we are done.
\end{proof}
Given a vector family $\mathcal V$, for any $i\in [n]$ we denote by $d_i$ the \textit{degree} of $i$, that is, the number of vectors from $\mathcal V$ that have a nonzero coordinate on position $i$.
\begin{prop}\label{prop4} For $n>k\ge 2$ the following holds: \\
(i) $F(3,2,0) =4$.\\
(ii) $F(n,k,k-2) = {n\choose k}$ if $n=k+2$ or $k\ge 3$ and $n=k+1$.\\
(iii) $F(n,k,k-2) = \max\bigl\{{n-k+2\choose 2}, k(n-k)+1\bigr\}$ for $n\ge k+3$.
\end{prop}
\begin{proof} The first part is very easy to verify. As for the other parts, assume that $u_i = -v_i$ for some $\mathbf u,\mathbf v\in \mathcal V$. Then $S(\mathbf u) = S(\mathbf v)$ and $\mathbf u$ and $\mathbf v$ must agree in the remaining coordinate positions. Thus, $d_i=2$ in this case. Since $F(k,k,k-2) =2, $ for $n=k+1$ a vertex of degree $2$ would force $|\mathcal V|\le 2+2 = 4.$ On the other hand, if we assume that $\mathcal V$ is homogeneous, then $|\mathcal V|\le {k+1\choose k}$ for $k\ge 3$, as desired. Since ${k+2\choose k}-{k+1\choose k}\ge 2$ for $k\ge 2$, the same argument works for $n=k+2, k\ge 2$ as well.

If $n>k+2$, the same argument as above implies that the family $\{S(\mathbf u):\mathbf u\in \mathcal V\}$ is $(k-2)$-intersecting. If $|\mathcal V|$ is maximal, then it is homogeneous. The bound then follows from the Complete Intersection Theorem (see Section \ref{sec3}).
\end{proof}

\subsection{Results}\label{sec22}

 Our main result concerning $\{0,\pm 1\}$-vectors is the following theorem, which determines $F(n,k,l)$ for all $k,l$ and sufficiently large $n$ and which shows the connection of $F(n,k,l)$ with the above stated theorems of Katona and Kleitman, as well as the problem of determining $g(n,k',l')$.

\begin{thm}\label{thm1} For any $k$ and $n\ge n_0(k)$ we have \\
 \begin{align*}1.\ \ \ F(n,k,l) =\ &{n-l\choose k-l} \ \ \ \ \ \ \ \ \ \ \ \ \text{for }\ 0\le l\le k.\\
 2. \ F(n,k,-l) =\ &f(k,l) {n\choose k} \ \ \ \ \ \ \ \ \ \ \ \ \ \ \ \ \ \ \ \ \  \text{for even }\ 0\le l\le k,\\
 3. \  F(n,k,-l) =\ &g\Big(n,k-\frac{l+1}2,\frac{l+1}2\Big)+f(k,l-1) {n\choose k} \ \ \text{for odd }\ 0\le l\le k.\end{align*}\end{thm}
Note that $f(k,0) = 1$ and so the values for $l=0$ in part 1 and $l=0$ in part 2 coincide. We remark that the statement of part 2 of the theorem for $l=k$ is obvious and for $l=k-1$ it was already derived in Proposition \ref{prop1}. Using Theorem~\ref{fkjcta}, we get the exact value of $F(n,k,-1)$. The value of $F(n,k,-l)$ for other odd $l$ is determined up to an additive term of order $O(n^{k-1})$ using Theorem~\ref{fkbar}.

Using the same technique, we may extend the result of parts 2 and 3 of Theorem \ref{thm1} in the following way:

\begin{thm}\label{thm2} Let $\mathcal V\subset \mathcal L_k$ be the set of vectors such  that for any $\mathbf v,\mathbf w\in \mathcal V$ we have $\langle\mathbf v,\mathbf w\rangle\ne -l-1$ for some $0\le l< k$. Then we have
 $$\max_{\mathcal V}|\mathcal V| = f(k,l) {n\choose k}+O(n^{k-1}).$$

\end{thm}
\bigskip

For $k = 3$ the values not covered by Propositions \ref{prop1}, \ref{prop3}, \ref{prop4} and Theorem \ref{thm1} are $F(n,3,0)$ and $F(n,3,-1)$ for $n$ not too large. We determine $F(n,3,0)$ in the following theorem, but first we need some preparation.

Let us use the notation $(a,b,c)$ for a set $\{a,b,c\}$ if we know that $a<b<c$. Also for $(a,b,c)\subset [n]$ let $\mathbf u(a,b,c) = (u_1,\ldots, u_n)$ with $u_i = 1$ for $i \in \{a,b,c\}$ and $u_i = 0$ otherwise. Further let $\mathbf v(a,b,c) = (v_1,\ldots, v_n)$ be the vector with $v_a = v_b = 1$, $v_c = -1$ and $v_i = 0$ otherwise.

Let us show two lower bounds for $F(n,3,0)$. Taking all non-negative vectors gives
$$F(n,3,0) \ge {n\choose 3}.$$
The other one is based on the following family:
$$\mathcal V(n) = \{\mathbf u(a,b,c),\mathbf v(a,b,c): |(a,b,c)\cap[3]|\ge 2\}.$$
Note that $\langle\mathbf v,\mathbf v'\rangle\ge 0$ for all $\mathbf v,\mathbf v'\in \mathcal V(n)$ and the first few values of $|\mathcal V(n)| $ are $|\mathcal V(3)| = 2, |\mathcal V(4)| = 8, |\mathcal V(5)| = 14, |\mathcal V(6)| = 20$.

\begin{thm}\label{thm3} We have \\
(1) $F(n,3,0) = |\mathcal V(n)|$ for $n = 3,4,5$.\\
(2) $F(6,3,0) = 21.$\\
(3) $F(n,3,0) = {n\choose 3}$ for $n\ge 7$.
\end{thm}
As we will see, the proof of $(2)$ is the most difficult.

\section{Cross-intersecting families}\label{sec3}
We say that two families $\mathcal A,\mathcal B\subset 2^{[n]}$ are \textit{$s$-cross-intersecting}, if for any $A\in\mathcal A, B\in \mathcal B$ we have $|A\cap B|\ge s$.
In the proof of Theorem \ref{thm1} we need to estimate the sum of sizes of two $t$-intersecting families that are $s$-cross-intersecting. For the proof some relatively crude  bounds are sufficient but we believe that this problem is interesting in its own right. As a matter of fact the case of non-uniform families was solved by Sali \cite{S} (cf. \cite{F2} for an extension with a simpler proof).

To state our results for the $k$-uniform case let us make some definitions.

\begin{opr} For $k\ge s>t\ge 1$ and $k\ge 2s-t$ define  the Frankl-family $\mathcal A_i(k,s,t)$ by $$\mathcal A_i(k,s,t) = \{A\subset{[k]\choose s}: |A\cap [t+2i]|\ge t+i\},\ \ \ 0\le i\le s-t.$$
\end{opr}
Note that for $A,A'\in \mathcal A_i(k,s,t)$ one has $|A\cap A'\cap [t+2i]|\ge t,$ in particular, $\mathcal A_i(k,s,t)$ is $t$-intersecting. Also, for $i<s-t$ every vertex of $[k]$ is covered by some $A\in \mathcal A_i(k,s,t)$.

The following result was conjectured by Frankl \cite{F1} and nearly 20 years later proved by Ahlswede and Khachatrian.

\begin{thm*}[Complete Intersection Theorem (\cite{AK})]\label{thmak} Suppose that $\mathcal A\subset{[k]\choose s}$ is $t$-intersecting, $k\ge 2s-t$. Then
\begin{equation}\label{eqstar} |\mathcal A|\le \max_{0\le i\le s-t}|\mathcal A_i(k,s,t)|=:m(k,s,t)\ \ \ \ \ \text{holds.}
\end{equation}
Moreover, unless $k = 2s, t=1$ or $\mathcal A$ is isomorphic to $\mathcal A_i(k,s,t)$, the inequality is strict.
\end{thm*}

\begin{opr} For $k\ge 2s-t$, $0\le i<s-t$ define $$\mathcal M_i(k,s,t) = \{A\subset [k]: |A|\ge s, |A\cap [t+2i]|\ge t+i\}\cup \{A\subset [k]:|A| \ge k-s+t\}.$$
\end{opr}
Note that $\{A\in {[k]\choose s}:A\in \mathcal M_i(k,s,t)\} = \mathcal A_i(k,s,t)$ for $k>2s-t$ and that $\mathcal M_i(k,s,t)$ is $t$-intersecting.

For fixed $k,s,t$ and $0\le i<s-t$ let us define the pair $\mathcal A_i = \{A\in {[n]\choose k}: A\cap [k]\in \mathcal M_i(k,s,t)\}$, $\mathcal B_i = \{[k]\}$. Then these non-empty $t$-intersecting families are $s$-cross-intersecting.

Note that $\mathcal A_{s-t}(k,s,t) = {[2s-t]\choose s}.$ For $i=s-t$ we define
\begin{align*}\mathcal A_{s-t} = &\{A\in {[n]\choose k}: |A\cap [2s-t]|\ge s\}\ \ \ \ \text{and}\\ \mathcal B_{s-t} = &\{B\in{[n]\choose k}: [2s-t]\subset B\}.
\end{align*}
Again the non-empty $t$-intersecting families $\mathcal A_{s-t}$ and $\mathcal B_{s-t}$ are $s$-cross-intersecting.\\

With this terminology we prove

\begin{thm}\label{thm4} Let $k>s>t\ge 1$ be integers, $k\ge 2s-t$. Suppose that $\mathcal A,\mathcal B\subset{[n]\choose k}$ are non-empty $t$-intersecting families which are cross $s$-intersecting. Then for $n\ge n_0(s,t,k)$ we have
\begin{equation}\label{eqthm4} \max_{\mathcal A,\mathcal B}\big\{|\mathcal A|+|\mathcal B|\big\}= \max_{0\le i\le s-t} \big\{|\mathcal A_i|+|\mathcal B_i|\big\}.\end{equation}
Moreover, unless $k = 2s, t=1$ or $\mathcal A,\mathcal B$ are isomorphic to $\mathcal A_i,\mathcal B_i$, the equality above transforms into a strict inequality.
\end{thm}

We note that the conclusive result for $t=0$ was obtained by the authors in the paper \cite{FK4}, and therefore we do not consider this case here.\\

\textbf{Remark.} The assumption $k\ge 2s-t$ is in fact not restrictive, since if  $A\in\mathcal A$ and $B\in\mathcal B$ then any two sets from $\mathcal A$ intersect in at least $2s-k$ elements inside $B$ (and analogously for $\mathcal B$ and $A$). If  $2s-k\ge t$, then the $t$-intersecting condition is implied by the $s$-cross-intersecting condition, and it reduces to the case studied in \cite{FK4}.

We also note that the problem makes sense only in case $s>t$. Otherwise, we may just take both $\mathcal A$ and $\mathcal B$ to be the same $t$-intersecting family of maximal cardinality. Thus, in case $t\ge s$ the problem reduces to a trivial application of the Complete Intersection Theorem.

\section{Shifting for vectors}

In the proof of Theorem \ref{thm1} we use two types of shifting. The first one pushes bigger coordinates to the left. For a given  pair of indices $i<j\in [n]$ and a vector $\mathbf v = (v_1,\ldots, v_n)\in \{0,\pm 1\}^n$ we define an $(i,j)$-shift $S_{i,j}(\mathbf v)$ of $\mathbf v$ in the following way. If $v_i\ge v_j$, then $S_{i,j}(\mathbf v) = \mathbf v$. If $v_i<v_j$, then $S_{i,j}(\mathbf v) := (v_1,\ldots, v_{i-1},v_j,v_{i+1},\ldots,v_{j-1},v_i,v_{j+1},\ldots, v_n),$ that is, it is obtained from $\mathbf v$ by interchanging its $i$-th and $j$-th coordinate.

Next, we define an $(i,j)$-shift $S_{i,j}(\mathcal W)$ for a family of vectors $\mathcal W\subset \{0,\pm 1\}^n$:

$$S_{i,j}(\mathcal W) := \{S_{i,j}(\mathbf v): \mathbf v\in \mathcal W\}\cup \{\mathbf v: \mathbf v,S_{i,j}(\mathbf v)\in \mathcal W\}.$$

The second is the \textit{up-shift}:
$S_i(v_1,\ldots,v_n)= (v_1,\ldots,v_n)$ if $v_i = 0$ or $1$ and $S_i(v_1,\ldots,v_n) = (v_1,\ldots,v_{i-1},1,v_{i+1},\ldots,v_n)$ if $v_i  = -1$. The shift $S_i(\mathcal W)$ is defined similarly to $S_{i,j}(\mathcal W)$:
$$S_{i}(\mathcal W) := \{S_{i}(\mathbf v): \mathbf v\in \mathcal W\}\cup \{\mathbf v: \mathbf v,S_{i}(\mathbf v)\in \mathcal W\}.$$

We call a system $\mathcal W$ \textit{shifted}, if $\mathcal W = S_{i,j}(\mathcal W)$ for all $i<j\in [n]$ and $\mathcal W = S_i(\mathcal W)$ for all $i\in[n]$. Any  system of vectors may be made shifted by means of a finite number of $(i,j)$-shifts and $i$-up-shifts. Moreover, it was shown in  \cite{FK} and \cite{Kl} (and is easy to verify directly) that $(i,j)$-shifts and up-shifts do not decrease the minimal scalar product in $\mathcal W$.

\section{Proof of Theorems \ref{thm1} and \ref{thm2}}\label{sec4}

%Therefore, we may assume that all the family $\mathcal V$ we work with is shifted.

\subsubsection*{Part 1 of Theorem \ref{thm1}}
First we show that the left hand side is at least the right hand side.
Take the family of vectors $\mathcal V = \{\mathbf v = (v_1,\ldots,v_n): v_1 = \ldots = v_{l} =1, v_i\in\{0,1\}\}$. It clearly satisfies the condition $\langle\mathbf v,\mathbf w\rangle \ge l$ for any $\mathbf v, \mathbf w\in \mathcal V$ and has cardinality  ${n-l\choose k-l}$.\\

We proceed to the upper bound. Take any family $\mathcal V$ of $\{0,\pm 1\}$-vectors having minimal scalar product at least $l$ and define its subfamilies  $\mathcal V_i(\epsilon) =\{\mathbf w\in\mathcal V:w_i = \epsilon\}$ for $\epsilon =\pm 1$.

\begin{cla}\label{cl10}
If for some $i$ both $\mathcal V_i(-1)$ and $\mathcal V_i(1)$ are nonempty, then
$$|\mathcal V_i(1)|+|\mathcal V_i(-1)| \le 2^{k}{k-1\choose l+1}{n-l-2\choose k-l-2} = O(n^{k-l-2}).$$
\end{cla}
We call any such $i\in[n]$ \textit{bad}.

\begin{proof} Consider the set families $\mathcal A^+ = \{S\in{[n]-\{i\}\choose k-1}: \exists \mathbf w\in \mathcal V_i(1) \text{ such that } S(\mathbf w) = S\cup \{i\}\}$ and $\mathcal A^- = \{S\in{[n]-\{i\}\choose k-1}: \exists \mathbf w\in \mathcal V_i(-1) \text{ such that } S(\mathbf w) = S\cup \{i\}\}$. Both $\mathcal A^+, \mathcal A^-$ are nonempty. Each of them is $(l-1)$-intersecting, moreover, they must be $(l+1)$-cross-intersecting, since otherwise we would have two vectors in $\mathcal V$ that have scalar product at most $l-1$.

Let $A\in\mathcal A^+$, $B\in \mathcal A^-$. Then $|A\cap B|\ge l+1$ for all $B\in \mathcal A^-$ implies $|\mathcal A^-|\le {|A|\choose l+1}{n-l-2\choose k-l-2}$ and similarly $|\mathcal A^+|\le {|B|\choose l+1}{n-l-2\choose k-l-2}.$

The statement of the claim follows from the trivial inequalities $|\mathcal V_i(1)|\le 2^{k-1}|\mathcal A^+|,$ $|\mathcal V_i(-1)|\le 2^{k-1}|\mathcal A^-|$.
\end{proof}

Consider the set $I\subset[n]$ of all bad coordinate positions and define a subfamily $\mathcal V_b$ of all vectors that have non-zero values on the bad positions: $\mathcal V_b = \{\mathbf v\in \mathcal V: \exists i\in I: v_i\ne 0\}$. Put $\mathcal V_g = \mathcal V-\mathcal V_b$. Since $\mathcal V_g$ is homogeneous on $[n]-I$ and is $l$-intersecting, we have $|\mathcal V_g|\le {n-|I|-l\choose k-l}$ for  $n\ge k^2$ due to the result of \cite{F1}. Therefore,
$$|\mathcal V|\le {n-|I|-l\choose k-l}+|I|O(n^{k-l-2})\le$$$$\le {n-l\choose k-l}-|I|\Bigl({n-|I|-l\choose k-l-1}-O(n^{k-l-2})\Bigr) \le {n-l\choose k-l},$$
for $n$ large enough, with the last inequality strict in case $|I|>0$.

\subsubsection*{Parts 2, 3  of Theorem \ref{thm1} and Theorem \ref{thm2}}

First we show that there is such a family of vectors for which the inequality is achieved.
For each $S\subset {n\choose k}$ we take  in $\mathcal V$ all the vectors from $\mathcal L_k(S)$ that have at most $l/2$ $-1$'s in case $l$ is even and the sets that have at most $(l-1)/2$ $-1$'s, together with the family in $V(n,k-(l+1)/2,(l+1)/2)$ that realizes the value of $g(n,k-(l+1)/2,(l+1)/2)$. It is not difficult to see that $\mathcal V$ satisfies the requirements of both Theorems \ref{thm1} and \ref{thm2} on the size and on the scalar product.\\

Now we prove the upper bound. The proof is somewhat similar to the previous case. Take any vector family satisfying the conditions of Theorem \ref{thm1} or Theorem \ref{thm2}.
For any set $X\subset [n]$ and any vector $\mathbf w$ of length $n$ denote by $\mathbf w_{|X}$ the restriction of $\mathbf w$ to the coordinates $i,i\in X$. Obviously $\mathbf w_{|X}$ is a vector of length $|X|$. We look at all vectors $\mathbf v\in \{\pm 1\}^{l+1}$ and split them into groups of antipodal vectors $\mathbf v$, $\bar{\mathbf  v}$, that is, vectors that satisfy $\langle\mathbf v,\bar{\mathbf v}\rangle = -l-1$.  Fix a set of coordinates $I = \{i_1,\ldots, i_{l+1}\}$ and define  $\mathcal V(I,\mathbf v) =\{\mathbf w\in\mathcal V:\mathbf w_{|I} = \mathbf v\}.$

%W.l.o.g. we assume that $\mathbf v$ are chosen in such a way that the number of $-1$'s in each is at least $(l+1)/2$.

\begin{cla}\label{cla10}
If both $\mathcal V(I,\mathbf v)$ and $\mathcal V(I, \bar{\mathbf  v})$ are nonempty, then
$$|\mathcal V(I,\mathbf v)|+|\mathcal V(I,\bar{\mathbf  v})| \le 2^{k-l-1}(k-l){n-l-2\choose k-l-2} = O(n^{k-l-2}).$$
\end{cla}
We call any such $I$ \textit{bad}. In particular, the claim gives that the total cardinality of all such ``bad'' subfamilies is $O(n^{k-1})$.

\begin{proof} Consider the set families $\mathcal A = \{S\in{[n]-I\choose k-l-1}: \exists \mathbf w\in \mathcal V(I,\mathbf v) \text{ such that } S(\mathbf w) = S\cup I\}$ and $\bar{\mathcal A} = \{S\in{[n]-I\choose k-l-1}: \exists \mathbf w\in \mathcal V(I,\bar{\mathbf v}) \text{ such that } S(\mathbf w) = S\cup I\}$. Both $\mathcal A, \bar{\mathcal A}$ are nonempty. Moreover, they must be cross-intersecting, since otherwise we would have two vectors in $\mathcal V$ that have scalar product exactly $-l-1$. Therefore, by Hilton-Milner theorem,
$$|\mathcal A|+|\bar{\mathcal A}|\le {n-l-1\choose k-l-1}-{n-k\choose k-l-1}+1\le (k-l){n-l-2\choose k-l-2}.$$
The statement of the claim follows from the trivial inequalities $|\mathcal V(I,\mathbf v)|\le 2^{k-l-1}|\mathcal A|,$ and $|\mathcal V(I,\bar{\mathbf v})|\le 2^{k-l-1}|\bar{\mathcal A}|$.
\end{proof}

Define $\mathcal V_b = \{\mathbf w\in\mathcal V:\exists I\subset {[n]\choose l+1}: \mathcal V(I,\mathbf v_{|I})\ne\emptyset, \mathcal V(I,\bar{\mathbf v}_{|I})\ne\emptyset\}$. Consider the  remaining family of ``good vectors'' $\mathcal V_g = \mathcal V-\mathcal V_b$.

First we finish the proof of Theorem \ref{thm2}. The family $\mathcal V_g$ is good in the following sense. If we fix $S\in{[n]\choose k}$ and consider the family of sets $\mathcal A = \{N(\mathbf w): w\in\mathcal V_g(S)\}$, then $\mathcal A$ satisfies $A_1\bigtriangleup A_2\le l$ for any $A_1,A_2\in \mathcal A$. Therefore, $|\mathcal V_g(S)|\le f(k,l)$ by Kleitman's theorem  and
$$|\mathcal V| \le \sum_{S\in{[n]\choose k}}f(k,l)+|\mathcal V_b| = f(k,l){n\choose k}+O(n^{k-1}).$$

The proof of Theorem \ref{thm2} is complete.

From now on we suppose that the family $\mathcal V$ is shifted. Note that, since shifting may increase scalar products, it is impossible to use it in the case of Theorem \ref{thm2}.  It is not difficult to see that the family $\mathcal V_g$ is up-shifted, since if we deleted some $\mathbf w\in \mathcal V$, then we would have deleted all $\mathbf u\in \mathcal V$ with $S(\mathbf u) = S(\mathbf w), N(\mathbf u)\supset N(\mathbf w)$. One can see in a similar way that $\mathcal V_g$ is shifted.

\begin{cla} For every $S\in{[n]\choose k}$ the family $\mathcal V_g(S)$ satisfies $N(\mathbf w)\cup N(\mathbf v)\le l$ for all $\mathbf v,\mathbf w\in \mathcal V_g(S)$.
\end{cla}
\begin{proof} Suppose for contradiction $N(\mathbf w)\cup N(\mathbf u)\ge l+1$. Consider any set $I\subset N(\mathbf w)\cup N(\mathbf u)$, $|I|= l+1$. Since $\mathcal V_g$ is up-shifted, there is a vector $\mathbf  u'$ in $\mathcal V_g(S)$ such that $\langle\mathbf w_{|I},\mathbf u'_{|I}\rangle = -l-1,$ which means that for $\mathbf v = \mathbf w_{|I}$ both $\mathcal V_g(I,\mathbf v)$ and $\mathcal V_g(I,\bar{\mathbf v})$ are nonempty, a contradiction.
\end{proof}

Let us first deal with the case of even $l$.

\begin{cla}\label{cla121} Let $l$ be even. If $I\subset {[n]\choose l+1}$ is bad, then for at least ${n-k\choose k-l-1}$ sets $S\supset I$, $S\in{[n]\choose k}$, we have $|\mathcal V_g(S)|\le f(k,l)-1$.
\end{cla}
\begin{proof} Consider the family $\mathcal A\subset 2^{S}$ of subsets of $S$, that is defined in the following way: $\mathcal A = \{N(\mathbf w)\cap S: \mathbf w\in \mathcal V_g(S)\}$. In view of the uniqueness part of Katona's theorem, it is sufficient to show that $\mathcal A$ does not contain one of the sets from the extremal family $\mathcal U^{l}$ for at least ${n-k\choose k-l-1}$ choices of $S$.

If $I$ is bad then  there exists a vector $\mathbf v$ of length $l+1$, such that both $\mathcal V(I,\mathbf v),\mathcal V(I,\bar{\mathbf v})$ are nonempty. Assume w.l.o.g. that $|N(\mathbf v)|\le l/2$ and take a vector $\mathbf w\in \mathcal V$ such that $\mathbf w|_I = \bar {\mathbf v}$. Then for any $S$ such that $S\cap S(\mathbf w)=I$ the set $N(\mathbf v)$ is missing from $\mathcal A$ (and, consequently, $|\mathcal V_g(S)|\le f(k,l)-1$). There are exactly ${n-k\choose k-l-1}$ such sets.
\end{proof}
Assume now that there are $t$ bad sets $I\subset {[n]\choose l+1}$. Then the number of sets $S\subset {[n]\choose k}$ that contain one of the bad sets $I$ is at least $t{n-k\choose k-l-1}/{k\choose l+1}$. Therefore, by Claims~\ref{cla10} and~\ref{cla121} we have

$$|\mathcal V|-f(k,l){n\choose k}\le -t\frac{{n-k\choose k-l-1}}{{k\choose l+1}}+\sum_{\text{bad }I}\frac 12\sum_{v\in\{\pm1\}^{l+1}}\Big(|\mathcal V(I,\mathbf v)|+|\mathcal V(I,\bar{\mathbf v})|\Big)\le$$$$\le -t\Bigl(\frac{{n-k\choose k-l-1}}{{k\choose l+1}}-2^k(k-l){n-l-2\choose k-l-2}\Bigr)<0,$$
provided $n>2^k k^2{k\choose l+1}$. We note that taking $n>4^kk^2$ makes the choice of $n$ from which the proof works independent of $l$.\\

The case of odd $l$ turns out to be harder. We shall need the following variant of Katona's theorem.

\begin{thm}\label{thmkattron}
  Assume that $\ff\subset 2^{[n]}\setminus {[n]\choose m+1}$ and for any $F,G\in \ff$ we have $|F\cup G|\le 2m+1$. Then $|\ff|\le \sum_{i=0}^m{n\choose i}$, moreover, for $n>2m+2$ the only example attaining the bound is $\bigcup_{i=0}^m{[n]\choose i}$.
\end{thm}
\begin{proof}
  W.l.o.g., we may assume that $\ff$ is shifted (we discuss the effect of this assumption on the uniqueness at the end of the proof). The proof is by induction. The statement is clear for $m=0$, moreover, the extremal family is unique. For $n=2m+2$, it is easy to see that $2^{[n]}\setminus {[n]\choose m+1}$ splits into pairs of complementary sets, which implies the statement.

  Assume that the statement holds for $(n-1,m)$, $(n-1,m-1),$  and let us prove it for $(n,m)$, $n>2m+2$. We have $|\ff| = |\ff(n)|+|\ff(\bar n)|$. By induction, we have $|\ff(\bar n)|\le \sum_{i=0}^{m}{n-1\choose i}$. Moreover, by shiftedness it follows that $|F\cup G|\le 2m-1$ for any $F,G\in \ff(n)$, and thus $|\ff(n)|\le \sum_{i=0}^{m-1}{n-1\choose i}$. Since $n-1>2(m-1)+1$, this inequality is sharp unless $\ff(n) = \bigcup_{i=0}^{m-1}{[n-1]\choose i}$. From here it should be clear that $\ff(\bar n) \subset \bigcup_{i=0}^m{[n-1]\choose i}$ and thus $\ff(\bar n) = \bigcup_{i=0}^m{[n-1]\choose i}$ in case of equality.

  We remark that if $\ff$ was not shifted initially, then it could not shift into $\bigcup_{i=0}^m{[n]\choose i}$, thus the uniqueness part holds for not shifted families as well.
\end{proof}

Let us return to the case of odd $l$. Consider the subfamily $\mathcal V'\subset\mathcal V$ of all vectors from $\mathcal V$ that have exactly $\frac {l+1}2$ minus ones, and put $\mathcal V'':=\mathcal V\setminus \mathcal V'$. Arguing as in the case of even $l$, but applying Theorem~\ref{thmkattron}, we get that $|\mathcal V''|\le \sum_{i=0}^{\frac{l-1}2} {k\choose i}{n\choose k}$, and the inequality is sharp if $\mathcal V''\ne \mathcal U$, where $\mathcal U$ consists of all $\{-1,0,1\}$-vectors with $k$ non-zero coordinates and at most $\frac{l-1}2$ minus ones.

Note also that any vector with $\frac {l+1}2$ minus ones has scalar product at least $-l$ with any vector from $\mathcal U$. Thus, it is clear that $\mathcal V'$ must avoid scalar product $-l-1$ and this is sufficient. Therefore, $|\mathcal V'|\le g(n,k-\frac{l+1}2,\frac{l+1}2)$ and, moreover, the largest $\mathcal V$ satisfying the requirements has size $g(n,k-\frac{l+1}2,\frac{l+1}2)+\sum_{i=0}^{\frac{l-1}2} {k\choose i}{n\choose k} = g(n,k-\frac{l+1}2,\frac{l+1}2)+f(k,l-1){n\choose k}$.

\section{Proof of Theorem \ref{thm3}}

Let $\mathcal W\subset \{0,\pm 1\}^n$ satisfy $\langle\mathbf w,\mathbf w\rangle = 3$, $\langle\mathbf w,\mathbf v\rangle\ge 0$ for all $\mathbf w,\mathbf v\in\mathcal W$ and suppose that $\mathcal W$ is shifted.

\begin{cla} If $S(\mathbf w) = (a,b,c)$ for $\mathbf w\in \mathcal W$ then either $\mathbf w = \mathbf u(a,b,c)$ or $\mathbf w = \mathbf v(a,b,c)$ holds.
\end{cla}
 \begin{proof} Assume the contrary and let $(x,y,z)$ be the non-zero coordinate values of the vector $\mathbf w$ (at positions $a,b,c$). Using the up-shift w.l.o.g. $(x,y,z) = (-1,1,1)$ or $(1,-1,1)$. In both cases by shifting $(1,1,-1)$ is also in $\mathcal W$. However it has scalar product $-1$ with both, a contradiction.
 \end{proof}

 From the claim $F(3,3,0)\le 2$, and $F(4,3,0)\le {4\choose 3}\times 2 = 8$ follow.

 From now on  $n\ge 5$ and we use induction on $n$ to prove the statement of the theorem. Set $\mathcal G = \{(a,b)\in{[n-1]\choose 2}: \mathbf u(a,b,n)\in \mathcal W\}$, $\mathcal H = \{(a,b)\in{[n-1]\choose 2}: \mathbf v(a,b,n)\in \mathcal W\}$. By shiftedness $\mathcal H\subset \mathcal G$ and by $\langle\mathbf v,\mathbf w\rangle\ge 0$, $\mathcal G$ and $\mathcal H$ are cross-intersecting. We may assume that $\mathcal H$ is non-empty, since otherwise $|\mathcal W|\le {n\choose 3}$ holds due to shiftedness.

 \begin{cla}\label{cl15} If $\mathbf v(2,3,n)\in\mathcal W$ then $|\mathcal H|  = |\mathcal G| = 3$.
 \end{cla}
\begin{proof} By shiftedness $\{(1,2),(1,3),(2,3)\}\subset \mathcal H\subset\mathcal G$. The statement follows from the fact that no other 2-element set can intersect those three sets.
\end{proof}

Since $\mathcal W(\bar n) \overset{def}{=}\{\mathbf w\in\mathcal W: w_n = 0\}$ satisfies $|\mathcal W(\bar n)|\le F(n-1,3,0)$, $\mathbf v(2,3,n)\in\mathcal W$ and Claim \ref{cl15} imply
\begin{align*}&|\mathcal W|\le F(n-1,3,0)+6, \ \ \ \ \ \ \ \ \text{which gives}\\
&|\mathcal W|\le 8+6 = 14\ \  \qquad \qquad \ \ \ \ \ \ \text{for }\ n = 5,\\
&|\mathcal W|\le 14+6 <21\  \ \ \qquad  \ \ \text{for }\ n = 6,\ \text{ and}\\
&|\mathcal W|\le {n-1\choose 3}+6 < {n\choose 3}\ \ \ \ \ \ \ \text{for }\ n \ge 7.
\end{align*}
Consequently, we may suppose that the degree of $n$ in $\mathcal W$ is at least $7$, in particular, $\mathcal H$ is the star, $\mathcal H = \{(1,2),(1,3),\ldots,(1,p)\}$ for some $p\le n-1$.

The following lemma is obvious.
\begin{lem}\label{lem16} One of the following holds.\\
(i)\ \ \ $p =2,\ |\mathcal H| = 1,\ |\mathcal G|\le (n-2)+(n-3) = 2n-5$\\
(ii)\ \ $p =3,\ |\mathcal H| = 2,\ |\mathcal G|\le (n-2)+1 = n-1$\\
(iii) $p \ge 4,\ |\mathcal H| \le |\mathcal G|\le n-2$.
\end{lem}
Since for $n = 5$ in all cases $|\mathcal H|+|\mathcal G|\le 6$ holds, the proof of part (1) of Theorem \ref{thm3} is complete.
Also, for $n\ge 7$ it follows that $|\mathcal H|+|\mathcal G|\le 2n-4$. Since ${n\choose 3}-{n-1\choose 3}={n-1\choose 2}>2n-4$ for $n\ge 7$, the induction step works fine in this case too.\\

The only case that remains is $n = 6$. Using shiftedness and $a = 1$ for  all $(a,b,n)$ with $\mathbf v(a,b,n)\in \mathcal W$ it follows that $a = 1$ holds for all $(a,b,c)$ with $\mathbf v(a,b,c)\in\mathcal W$.

Define $\mathcal B = \{(b,c)\in{[6]-\{1\}\choose 2}: \mathbf v(1,b,c)\in\mathcal W\}$ and $\mathcal D = \{(d,e,f)\in{[2,6]\choose 3}: \mathbf u(a,b,c)\in\mathcal W\}.$
Note that $(b,c)\in\mathcal B$ and $D\in \mathcal D$ cannot satisfy $D\cap\{b,c\} = \{c\}$, since otherwise $\langle\mathbf u(d,e,f),\mathbf v(1,b,c)\rangle=-1$. Let us note that one can put into $\mathcal W$ all the vectors $\mathbf u(1,b,c)$ with $2\le b<c\le 6$ because all the vectors with a $-1$ position have $1$ in the first coordinate.

Therefore $|\mathcal W| = 10+|\mathcal B|+|\mathcal D|$ holds. We have to prove
\begin{equation}\label{eq44}|\mathcal B|+|\mathcal D|\le 11.
\end{equation}
Let us first exhibit the system with $21$ vectors showing $F(6,3,0)\ge 21$.
$$\mathcal U_6 = \{\mathbf u(1,b,c),\mathbf v(1,b,c):(b,c)\subset[2,6]\}\cup\{\mathbf u(2,3,4)\}.$$
We have $|\mathcal U_6| = 2\times{5\choose 2}+1 = 21$ and the only non-trivial scalar product to check is that $\langle\mathbf v(1,b,c),\mathbf u(2,3,4)\rangle\ge 0$. It is true automatically if $c\ge 5.$ For $c = 3,4$ it follows that $b\in \{2,3\}$ making the scalar product equal to $0$ as desired.

To avoid a tedious case by case analysis it is simplest to give a matching from the $9$-element set $\{D\in{[2,6]\choose 3}: D\ne (2,3,4)\}$ into the $10$-element set $\{(b,c):2\le b<c\le 6\}$ such that whenever $D$ and $(b,c)$ are matched, $D\cap (b,c) = \{c\}$ holds. This will prove (\ref{eq44}). An example of the desired matching is exhibited below.
\begin{align*}
&(2,3,5)\ - \ (4,5)\ \ \ \ \ (2,4,5)\ - \ (3,5)\ \ \ \ \ (2,3,6)\ - \ (4,6)\\
&(2,4,6)\ - \ (5,6)\ \ \ \ \ (2,5,6)\ - \ (3,6)\ \ \ \ \ (3,4,5)\ - \ (2,3)\\
&(3,4,6)\ - \ (2,4)\ \ \ \ \ (3,5,6)\ - \ (2,5)\ \ \ \ \ (4,5,6)\ - \ (2,6)
\end{align*}

A careful analysis shows that $\mathcal U_6$ is the only extremal configuration. The proof is complete.

\section{Proof of Theorem \ref{thm4}}

%W.l.o.g. assume that $|\mathcal A|\ge |\mathcal B|$.

We may assume that the families $\mathcal A, \mathcal B$ are shifted and, thus, both contain $[k]$. We define $\mathcal A(T) = \{A\backslash T: A\in \mathcal A, A\supset T\}.$ Note that we had a similar notation for families of vectors, but now we are working with sets only, so it should not cause any confusion. We call a set $T\in {[n]\choose s}$ a \textit{kernel} for a set family $\mathcal S$, if $\mathcal S(T)$ contains $k+1$ pairwise disjoint edges.  An inequality proved in \cite{F3} states that
\begin{equation}\label{eq12}|\mathcal A(T)|\le k{n-s-1\choose k-s-1},\end{equation}
if $T$ is not a kernel for $\mathcal A$, and similarly for $\mathcal B(T)$.

The following claim is an easy application of the pigeon-hole principle.
\begin{cla}\label{cla12} Suppose that $C_0,\ldots, C_k$ form a sunflower with center  $T$, i.e. $C_i\cap C_j= T$ for all $0\le i<j\le k$. Suppose $D$ is a $k$-element set. Then there exists an $i$ such that $D\cap C_i = D\cap T$ holds.
\end{cla}
We immediately get the following corollary:

\begin{cor}\label{cor12} If $|T| = s$ and $|D\cap C_i|\ge s$ for all $0\le i\le k$, then $T\subset D$.
\end{cor}

\begin{lem}\label{lem12} If $T_1$ and $T_2$ are kernels for $\mathcal A$ and $\mathcal B$  respectively, then $T_1 = T_2$ holds.
\end{lem}
\begin{proof} From Corollary \ref{cor12} it follows that for any $B\in \mathcal B$ we have $T_1\subset B$. Applying it for $B_1,B_2\in \mathcal B$, $B_1\cap B_2 = T_2$, we get that $T_1\subset B_1\cap B_2 = T_2$, but since $|T_1| = |T_2|$, we have $T_1 = T_2$.
\end{proof}

From Corollary \ref{cor12} and Lemma \ref{lem12} it follows that if both $\mathcal A$ and $\mathcal B$ have kernels, then $|\mathcal A|+|\mathcal B|\le 2{n-s\choose k-s}$, which is smaller then the bound in (\ref{eqthm4}). From now on we may w.l.o.g. assume that $\mathcal B$ does not have a kernel and that $\mathcal B$ contains the set $[k]$ as an element. Then any kernel of $\mathcal A$ must be a subset of $[k]$. Let $\mathcal J\subset{[k]\choose s}$ be the family of kernels of $\mathcal A$.
Define $\tilde{\mathcal A} = \{A\in \mathcal A:\nexists T\in \mathcal J, A\supset T\}$.

\begin{cla}\label{cla21} We have $|\tilde{\mathcal A}| = O(n^{k-s-1})$. Analogously, we have $|\mathcal B| = O(n^{k-s-1})$.
\end{cla}
\begin{proof} Any set from $\tilde{\mathcal A}$ must intersect $[k]$ in at least $s$ elements. Thus, $|\tilde{\mathcal A}|\le \sum_{T\in{[k]\choose s}} |\tilde{\mathcal A}(T)|$, which by inequality (\ref{eq12}) is at most ${[k]\choose s} k{n-s-1\choose k-s-1} = O(n^{k-s-1})$. The same proof works for $\mathcal B$.
\end{proof}

Due to Claim \ref{cla12} and the fact that $\mathcal A$ is $t$-intersecting, for any set $A\in\mathcal A$ and any $T\in\mathcal J$ we have $|A\cap T|\ge t$. Moreover, repeating the proof of Lemma \ref{lem12}, it is easy to see that for any $T_1,T_2\in \mathcal J$ we have $|T_1\cap T_2|\ge t$. Therefore, $|\mathcal J|\le m(k,s,t)$ (see the formulation of the Complete Intersection Theorem). Due to Claim \ref{cla21}, we have
\begin{equation}\label{eq66}|\mathcal A|+|\mathcal B|\le |\mathcal J| {n-k\choose k-s}+O(n^{k-s-1}).\end{equation}
 Indeed, the only additional thing one has to note is that the number of sets from $\mathcal A$ that intersect $[k]$ in at least $s+1$ elements is $O(n^{k-s-1})$.

 On the other hand, it is easy to exhibit an example of a family $\mathcal A$ that is $t$-intersecting and is $s$-cross-intersecting with $\{[k]\}$, and which has cardinality $m(k,s,t){n-k\choose k-s}$. For that one just has to take a maximum $t$-intersecting family $\mathcal J'$ of $s$-element sets in $[k]$ and put $\mathcal A = \{A\in{[n]\choose k}: \exists T\in \mathcal J': A\cap [k]=T\}.$

 Therefore, if  $|\mathcal J|<m(k,s,t)$, then by (\ref{eq66}) and the previous construction $|\mathcal A|+|\mathcal B|$ cannot be maximal for large $n$. Therefore, from now on we may assume that $\mathcal J = \mathcal A_i(k,s,t)$ for some $0\le i\le s-t$, where $|\mathcal A_i(k,s,t)| = m(k,s,t)$.

We remark that out of $\mathcal A_j(k,s,t)$ the only family that satisfies $|\mathcal A_j(k,s,t)| = |\mathcal A_j(k+1,s,t)|$ is the family $\mathcal A_{s-t}(k,s,t)$. Indeed, it consists of all the sets that have all their $s$ elements among the first $2s-t$ elements and does not use elements $j$ with $j>2s-t$. For all the other families, as we have already mentioned, the degree of each $j\in[k]$ is positive.\\

Suppose first that $0\le i<s-t$ and that $|\mathcal A_{s-t}(k,s,t)|<m(k,s,t)$. Most importantly for us, it means that $m(k-1,s,t)<m(k,s,t)$ due to the discussion in the previous paragraph. Then it is easy to see that $|\mathcal B| = 1$. Indeed, if $B_1,B_2\in\mathcal B$, then due to Corollary \ref{cor12} we know that  $T\subset B_1\cap B_2$ for each $T\in\mathcal J$. Since $|B_1\cap B_2|\le k-1$, we have $|\mathcal J|\le m(k-1,s,t)<m(k,s,t)$.

We showed that the extremal pair of families satisfies $\mathcal B = \{[k]\}$. To complete the proof in this case and to show that  the pair of extremal families is one of $\mathcal A_i,\mathcal B_i$ for $0\le i<s-t$, we have to verify that $\{A\cap [k]:A\in \mathcal A\}\subset\mathcal M_i(k,s,t)$ provided that $\mathcal J = \mathcal A_i(k,s,t)$.

Take any set $A$ such that $A\notin\mathcal M_i(k,s,t),$ that is, $|A\cap [k]|<k-s+t$ and $|A\cap [t+2i]|<t+i$. The following claim completes the proof in the case $m(k,s,t)$ is attained on $\mathcal M_i(k,s,t)$ with $i<s-t$.
\begin{cla} In the notations above, there is $T\in \mathcal J = \mathcal A_i(k,s,t)$ such that $|T\cap A|\le t-1$.
\end{cla}
\begin{proof}
W.l.o.g., let $A = [1,l]\cup [m+1,k]$, where $l\le t+i-1$ and $l+m< k-s+t$. Set $T = [l',l'+s-1]$ with $l'\le i+1$ and with $l'$ chosen in such a way that $|T\cap A|$ is minimized. It is easy to see that, first, $T\in\mathcal J$ and, second, $|\mathcal T\cap A| = \max\{l-i,l+m+s-k\}<t$. \end{proof}\bigskip

If the only maximal family is $\mathcal A_{s-t}(k,s,t)$, then $\mathcal J = \mathcal A_{s-t}(k,s,t)$ and $\mathcal A$ cannot contain a set $A$ that satisfy $|A\cap [2s-t]|<s$. Indeed, then $A$ intersects one of the $T\in \mathcal J$  in less than $t$ elements, which by Claim \ref{cla12} means that there exists $A'\in \mathcal A$ such that $|A'\cap A|<t,$ a contradiction. Therefore, the family $\mathcal A$ is contained in $\mathcal A_{s-t}$, and we are only left to prove that $\mathcal B\subset \mathcal B_{s-t}$. Assume that $\mathcal B$ contains $B$, such that $B\nsupseteq [2s-t]$. But then there is $T\in\mathcal J$ such that $|B\cap T|\le s-1$, a contradiction.\\

The last case that remains to verify is when both $\mathcal A_{s-t}(k,s,t)$ and $\mathcal A_{s-t-1}(k,s,t)$ have size $m(k,s,t)$. But either $\mathcal J = \mathcal A_{s-t-1}(k,s,t)$ or $\mathcal J = \mathcal A_{s-t}(k,s,t)$ and we end up in one of the two scenarios described above.

It is clear from the proof that the families that we constructed are the only possible extremal families (excluding the case $n=2k,s=1$), which is mostly due to the uniqueness in the formulation of the Complete Intersection Theorem.\\

{\sc Acknowledgements. } We thank Danila Cherkashin and Sergei Kiselev for pointing out the error in the original version of Theorem~\ref{thm1}.

\end{document}